\RequirePackage{pgfplots}
\documentclass[sn-mathphys,pdflatex]{sn-jnl}
\usepackage{upquote}
\usepackage{multirow}
\jyear{2022}%
\newcommand{\RR}{\mathbb{R}}
\newcommand{\CC}{\mathbb{C}}
\newcommand{\rme}{\mathrm{e}}

\newcommand{\matlab}{\textsc{matlab}}
\newcommand{\bb}{\boldsymbol}
\newcommand{\bu}{\bb u}
\newcommand{\bU}{\bb U}
\newcommand{\bx}{\bb x}
\newcommand{\kP}{KronPACK}
\newtheorem{lemma}{Lemma}[section]
\newtheorem{definition}{Definition}[section]
\newtheorem{remark}{Remark}
\usepackage{pgfplots}
\pgfplotsset{compat=newest}
\usetikzlibrary{plotmarks}
\usetikzlibrary{arrows.meta}
\usepgfplotslibrary{patchplots}
\usepackage{grffile}
\usepackage{amsmath}

\begin{document}
\title[{\small A $\mu$-mode BLAS approach for multidimensional tensor-structured problems}]
      {A $\mu$-mode BLAS approach for multidimensional tensor-structured problems}

\author*[1]{\fnm{Marco} \sur{Caliari}}\email{marco.caliari@univr.it}
\author[2]{\fnm{Fabio} \sur{Cassini}}\email{fabio.cassini@unitn.it}
\equalcont{These authors contributed equally to this work.}
\author[3]{\fnm{Franco} \sur{Zivcovich}}\email{franco.zivcovich@sorbonne-universite.fr}
\equalcont{These authors contributed equally to this work.}
\affil[1]{\orgdiv{Department of Computer Science}, \orgname{University of Verona}, \country{Italy}}
\affil[2]{\orgdiv{Department of Mathematics}, \orgname{University of Trento}, \country{Italy}}
\affil[3]{\orgdiv{Laboratoire Jacques--Louis Lions}, \orgname{Sorbonne University}, \country{France}}


\abstract{In this manuscript, we present a common tensor framework
  which can be used to generalize one-dimensional numerical tasks to
  \textit{arbitrary}
  dimension $d$ by means of tensor product formulas.
This is useful, for example, in the context of multivariate interpolation,
multidimensional function approximation using pseudospectral expansions
and solution of stiff differential equations on tensor product domains.
The key point to obtain an efficient-to-implement BLAS
formulation  consists in the suitable usage of the $\mu$-mode product
(also known as tensor-matrix product or mode-$n$ product) and related
operations, {\color{black}such as the Tucker operator.
Their MathWorks MATLAB\textsuperscript{\textregistered}/GNU Octave
implementations are discussed in the paper, and collected in the
 package \kP.}
We present numerical results on {\color{black} experiments up to dimension six}
from different 
fields of numerical analysis, which show the effectiveness of the approach.}

\maketitle

\section{Introduction} \label{sec:intro}
Many one-dimensional tasks in numerical analysis can be generalized to a 
two-dimensional formulation by means of tensor product formulas. This is the 
case, for example, in the context of spectral decomposition or interpolation 
of multivariate functions.
Indeed, the one-dimensional formula
\begin{equation*}
s_{i}=\sum_{j=1}^{m}t_{j}
\ell_{ij}, \quad 1 \leq i \leq n,
\end{equation*}
where the values $t_j$ are linearly combined to obtain the values $s_i$
(i.e.~$\bb s = L \bb t$, with $\bb s = (s_i) \in \CC^n$, $\bb t = (t_j) \in \CC^m$, 
and $L = (\ell_{ij})\in \CC^{n\times m}$),
can be easily extended to the two-dimensional case as
\begin{equation} \label{eq:gen2dsum}
s_{i_1 i_2}={\color{black}\sum_{j_2=1}^{m_2}\sum_{j_1=1}^{m_1}}t_{j_1j_2}
\ell_{i_1j_1}^1\ell_{i_2j_2}^2, \quad 1 \leq i_1 \leq n_1, \quad 1 \leq i_2 \leq n_2.
\end{equation}
The meaning of the involved scalar quantities depends on the specific 
example under consideration.
In any case, a straightforward implementation of formula~\eqref{eq:gen2dsum} 
requires four nested for-loops, with a resulting computational cost of 
$\mathcal{O}(n^4)$ (if, for simplicity, we consider $m_1=m_2=n_1=n_2=n$). 
On the other hand, formula \eqref{eq:gen2dsum} can be written equivalently 
in matrix formulation as
\begin{equation}\label{eq:gen2dmatmat}
\bb S = L_1 \bb T L_2^{\sf T},
\end{equation}
where $L_1 = (\ell^1_{i_1 j_1})\in \CC^{n_1\times m_1}$,
$L_2 = (\ell^2_{i_2 j_2})\in \CC^{n_2\times m_2}$, 
$\bb T = (t_{j_1 j_2}) \in \CC^{m_1 \times m_2}$ 
and $\bb S = (s_{i_1 i_2}) \in \CC^{n_1 \times n_2}$.
The usage of formula \eqref{eq:gen2dmatmat} requires two separate 
matrix-matrix products as floating point operations, each of which can be 
implemented with three nested for-loops: this approach reduces then the 
cost of computing the elements of $\bb S$ to $\mathcal{O}(n^3)$.
On the other hand, a more efficient way to realize 
formula~\eqref{eq:gen2dmatmat} is to exploit optimized Basic Linear
Algebra Subprograms 
(BLAS) {\color{black}\cite{DDCHD90,mkl,XQY12,cublas}},
which are a set of numerical linear algebra routines 
that perform the just mentioned matrix operations with a level of efficiency 
close to the theoretical hardware limit.
{\color{black}
  A performance comparison of the three approaches to compute the values
  $s_{i_1i_2}$
  in \matlab{}\footnote{We refer to \matlab{} as the common
    language
interpreted by the softwares MathWorks MATLAB\textsuperscript{\textregistered}
and GNU Octave, for instance.} language,
for increasing size of the task, is given in Table~\ref{tab:mmp2drandmatlab}.}%
\begin{table}[ht!]
\centering
\begin{tabular}{|c|c|c|c|c|}
\hline
  & $n=50$ & $n=100$ & $n=200$ & $n=400$\\
 \hline
Nested for-loops & 1.8e-2 & 2.8e-1 & 4.8e0 & 8.0e1\\
Matrix-matrix products (for-loops) & 7.8e-4 & 5.5e-3 & 4.9e-2 & 3.9e-1\\
Matrix-matrix products (BLAS) & 2.1e-5 & 5.6e-5 & 1.7e-4 & 1.2e-3\\
\hline
\end{tabular}
\caption{Wall-clock time (in seconds) for the computation
  of the values $s_{i_1i_2}$ in formula \eqref{eq:gen2dsum} 
  with increasing size $m_1=m_2=n_1=n_2=n$ and different approaches,
  using MathWorks MATLAB\textsuperscript{\textregistered} R2019a.
  The input values are standard normal distributed random numbers.}
\label{tab:mmp2drandmatlab}
\end{table}
As expected, for all the values of $n$ under study, the most efficient way 
to compute the elements of $\bb S$ is {\color{black}realizing
  formula~\eqref{eq:gen2dmatmat}
  through the BLAS approach.}
Remark that the considerations on the complexity cost and BLAS efficiency 
are basically 
language-independent, and apply for other interpreted or compiled
languages as well, like \textsc{Python}, \textsc{Julia}, \textsc{R},
\textsc{Fortran}, and \textsc{C++}.
For clarity of exposition and simplicity of presentation of the codes, we will 
use in this manuscript, from now on, \matlab{} programming language.

In other contexts, such as numerical solution of (stiff) differential
equations on 
two-dimensional tensor product domains by means of exponential integrators or 
preconditioned iterative methods, it is required to compute quantities like
\begin{equation}\label{eq:gen2dkron}
\mathrm{vec}{(\bb S)} = (L_2 \otimes L_1) \mathrm{vec}{(\bb T)},
\end{equation}
being again $L_1$, $L_2$, $\bb T$ and $\bb S$ matrices of suitable size whose 
meaning depends on the specific example under consideration. Here
$\otimes$ denotes the standard Kronecker product of two matrices,
while $\mathrm{vec}$
represents the vectorization operator, see Appendix \ref{app:A} for their
formal {\color{black}definitions}.
A straightforward implementation of formula \eqref{eq:gen2dkron}
would need to assemble the large-sized matrix $L_2 \otimes L_1$. If, 
for simplicity, we consider again
$m_1=m_2=n_1=n_2=n$, this approach requires a storage and a 
computational cost of $\mathcal{O}(n^4)$, 
which is impractical. However, owing to the properties of the Kronecker
product 
(see Appendix~\ref{app:A}), we {\color{black}can see that
  formula~\eqref{eq:gen2dkron} is equivalent to
  formula~\eqref{eq:gen2dmatmat}}. Therefore, all the considerations
made for the previous example on the employment of optimized BLAS apply
also in 
this case.

{\color{black}The aim of this work is to provide a common framework for
  generalizing 
formula \eqref{eq:gen2dmatmat} in \textit{arbitrary} dimension $d$,
which will result in an efficient BLAS realization of the underlying task.
This is very useful in the context of solving tensor-structured problems
which may arise from different scientific and engineering fields.}
The pursued approach is illustrated in detail in Section~\ref{sec:mmp}, in which we
present the $\mu$-mode product and some associated operations (the Tucker 
operator, in particular), both from a theoretical and a practical point of view.
These operations are widely known by the tensor algebra community, but their
usage is mostly restricted in the context of tensor decompositions 
(see \cite{K06,KB09}).
Then, we proceed in Section~\ref{sec:1d2d} by describing more precisely
the one- and two-dimensional formulations of the
problems mentioned in this section, as well 
as their generalization to the $d$-dimensional case in terms of $\mu$-mode 
products.
We collect in Section~\ref{sec:code} the related
numerical experiments and we finally 
draw the conclusions in Section~\ref{sec:conc}.

All the functions and the scripts needed to perform the relevant
tensor operations and
to reproduce the numerical examples of this manuscript are contained
in our \matlab{} package
\kP\footnote{Freely available, under the MIT
  license, at \url{https://github.com/caliarim/KronPACK}}.
\section{The $\mu$-mode product and its applications} \label{sec:mmp}
In order to generalize formula \eqref{eq:gen2dmatmat} to the $d$-dimensional 
case, we rely on some concepts from tensor algebra (see \cite{K06,KB09} 
for more details).
Throughout this section,
we assume that $\bb T\in\CC^{m_1\times\cdots\times m_d}$
is an order-$d$ tensor whose elements are either denoted by $t_{j_1\ldots j_d}$ 
or by $\bb T(j_1,\ldots,j_d)$.
\begin{definition}
  A \emph{$\mu$-fiber} of $\bb T$ is a
  vector in $\CC^{m_\mu}$ obtained by fixing every index of the tensor but 
  the $\mu$th.
\end{definition}
A $\mu$-fiber is nothing but a generalization of the concept of rows and
columns of a matrix. Indeed, for an order-2 tensor (i.e.~a matrix), 
$1$-fibers are the columns, while $2$-fibers are the rows. 
On the other hand, for an
order-3 tensor, $1$-fibers are the column vectors, 
$2$-fibers are the row vectors while $3$-fibers are the so-called 
``page'' or ``tube'' vectors, which means vectors along the third dimension.
{\color{black}%
\begin{definition} \label{def:mumatric}
  The \emph{$\mu$-matricization} of $\boldsymbol{T}$,
denoted by $T^{(\mu)}\in~\CC^{m_\mu \times m_1\cdots m_{\mu-1}m_{\mu+1}\cdots m_d}$, is
defined as the matrix whose columns are the $\mu$-fibers of $\boldsymbol{T}$.
\end{definition}
Remark that for an order-2 tensor the $1$- and
$2$-matricizations simply correspond to the matrix itself and its transpose.
In {\color{black}dimensions} higher than two, the $\mu$-matricization requires the concept of
generalized transpose of a tensor and its unfolding  into a matrix.
The first operation is realized in \matlab{} by the \texttt{permute} function,
that we use to interchange $\mu$-fibers with 1-fibers of the tensor $\boldsymbol{T}$. 
The second operation is performed by the \texttt{reshape} function, that we use 
to unfold the ``transposed'' tensor into the matrix $T^{(\mu)}$.
In \matlab{}, the anonymous function which performs the $\mu$-matricization
of a tensor \texttt{T}, given
\begin{verbatim}

m = size(T);
d = length(m);

\end{verbatim}
can be written as
\begin{verbatim}

mumat = @(T,mu) reshape(permute(T,[mu,1:mu-1,mu+1:d]),...
                        m(mu),prod(m([1:mu-1,mu+1:d])));

\end{verbatim}
By means of $\mu$-fibers, it is possible to define the following operation.
\begin{definition} \label{def:mmp}
  Let $L\in\CC^{n\times m_\mu}$ be a matrix. The \emph{$\mu$-mode product}
  of $\boldsymbol{T}$ with $L$, denoted by $\bb S = \bb T \times_\mu L$, is the
  tensor $\boldsymbol S\in\CC^{m_1\times\cdots\times m_{\mu-1}\times n\times m_{\mu+1}\times\cdots \times m_d}$
  obtained by multiplying the matrix $L$ onto the $\mu$-fibers of $\boldsymbol{T}$.
\end{definition}
From this definition, it appears clear that the $\mu$-fiber
$\boldsymbol S(j_1,\ldots,j_{\mu-1},\cdot,j_{\mu+1},\ldots,j_d)$ of
$\boldsymbol S$ can be computed as the matrix-vector product of $L$ and
the $\mu$-fiber $\boldsymbol T(j_1,\ldots,j_{\mu-1},\cdot,j_{\mu+1},\ldots,j_d)$.
Therefore, the $\mu$-mode product $\boldsymbol{T}\times_\mu L$ might
be performed by calling $m_1\cdots m_{\mu-1}m_{\mu+1}\cdots m_d$ times
level 2 BLAS.
However, owing to the concept of matricization of a tensor
introduced in Definition~\ref{def:mumatric}, it is possible to
perform the same task more efficiently by using a single level 3 BLAS call.
Indeed, the $\mu$-mode 
product of $\boldsymbol{T}$ with  $L$ is just the 
tensor $\boldsymbol S$ such that 
\begin{equation}\label{eq:mmp}
    S^{(\mu)}=L T^{(\mu)}.
\end{equation}}%
In particular, in the two-dimensional setting, the $1$-mode product corresponds 
to the
multiplication $L\boldsymbol T$, while the $2$-mode product corresponds
to $(L\boldsymbol T^{\sf T})^{\sf T} = \boldsymbol TL^{\sf T}$. In general, we can
compute the matrix $S^{(\mu)}$ appearing in formula (\ref{eq:mmp}) as
\texttt{L*mumat(T,mu)}, and in order to recover the tensor
$\boldsymbol{S}$ from $S^{(\mu)}$ we need to invert the operations of unfolding
and ``transposing''. This can be done easily with the aid of the \matlab{} 
functions \texttt{reshape} and \texttt{ipermute}, respectively. 
All in all, given \verb+n = size(L,1)+, the anonymous function that computes the
$\mu$-mode product of an order-$d$ tensor $\bb T$ with $L$ by a single 
matrix-matrix product can be written as
\begin{verbatim}

mump = @(T,L,mu) ipermute(reshape(L*mumat(T,mu),...
                 [n,m([1:mu-1,mu+1:d])]),[mu,1:mu-1,mu+1:d]);

\end{verbatim}
{\color{black}Notice that from formula~\eqref{eq:mmp} it appears clear
  that the computational cost of the
  $\mu$-mode product, in terms of floating point operations, is
  $\mathcal{O}(nm_1\cdots m_d)$.}

  One of the main applications of the $\mu$-mode product is
  the so-called \emph{Tucker operator}, {\color{black} which is implemented
    in \kP{} in the function \verb+tucker+.}
\begin{definition} \label{def:tucker}
  Let $L_\mu\in\CC^{n_\mu\times m_\mu}$ be matrices, with $\mu=1,\ldots,d$.
  The \emph{Tucker operator} of $\boldsymbol{T}$ with $L_1,\ldots,L_d$ is the
  tensor
    $\boldsymbol S\in\CC^{n_1\times\cdots\times n_d}$ obtained by concatenating
    $d$ consecutive $\mu$-mode products with matrices $L_\mu$, that is
\begin{equation}\label{eq:tuckerop}
    \bb S = \boldsymbol{T}\times_1L_1\times_2\cdots\times_dL_d.
\end{equation}
\end{definition}%
We notice that the single element
$s_{i_1\ldots i_d}$ of $\boldsymbol{S}$ in formula
\eqref{eq:tuckerop} turns out to be
\begin{equation}\label{eq:nestedmumode}
  s_{i_1\ldots i_d}={\color{black}\sum_{j_d=1}^{m_d}\cdots \sum_{j_1=1}^{m_1}}
  t_{j_1\ldots j_d}\prod_{\mu=1}^d\ell_{i_\mu j_\mu}^\mu,\quad
  1\le i_\mu\le n_\mu{\color{black},}
\end{equation}
provided that $\ell^\mu_{i_\mu j_\mu}$ are the elements of $L_\mu$.
{\color{black}Hence, as formula~\eqref{eq:nestedmumode} is 
clearly the generalization of formula \eqref{eq:gen2dsum} to the 
$d$-dimensional setting, formula \eqref{eq:tuckerop} is the 
sought $d$-dimensional generalization of formula \eqref{eq:gen2dmatmat}. 
  We also notice that the Tucker operator~\eqref{eq:tuckerop}
  is invariant with respect to the ordering of the  $\mu$-mode products,
  and that the implicit ordering given by Definition~\ref{def:tucker}
  is equivalent to performing the sums
  in formula~\eqref{eq:nestedmumode} starting from the innermost.}

The Tucker operator is strictly connected with the Kronecker product of matrices 
applied to a vector.
\begin{lemma}\label{lem:krontomu}
  Let $L_\mu \in\CC^{n_\mu \times m_\mu}$ be matrices, with $\mu=1,\ldots,d$.
  Then,
    the elements of $\bb S$ in formula \eqref{eq:tuckerop} are equivalently
    given by
    \begin{equation}\label{eq:krontomu}
        \mathrm{vec}(\boldsymbol{S}) = 
        (L_d \otimes \cdots \otimes L_1)\mathrm{vec}(\boldsymbol{T}).
    \end{equation}
\end{lemma}
\begin{proof}
    The $\mu$-mode product satisfies the following property
    \begin{equation*}
      \boldsymbol{S}=\boldsymbol{T} \times_1 L_1 \times_2 \cdots \times_d L_d
      \iff S^{(\mu)} = L_\mu T^{(\mu)}(L_d\otimes \cdots \otimes L_{\mu+1}\otimes L_{\mu-1} \otimes \cdots \otimes L_1)^{\sf T}, 
    \end{equation*}
    see \cite{KB09}. Then, with $\mu=1$ we obtain
    \begin{equation*}
      \boldsymbol{S}=\boldsymbol{T} \times_1 L_1 \times_2 \cdots \times_d L_d
      \iff S^{(1)} = L_1 T^{(1)}(L_d\otimes \cdots \otimes L_2)^{\sf T}.
    \end{equation*}
    By means of the properties of the Kronecker
    product (see Appendix \ref{app:A}) we have then
    \begin{equation*}
      S^{(1)} =L_1 T^{(1)}(L_d\otimes \cdots \otimes L_2)^{\sf T} \iff
      \mathrm{vec}(S^{(1)}) = (L_d\otimes \cdots \otimes L_1)
      \mathrm{vec}( T^{(1)})
    \end{equation*}
    and finally, by definition of vec operator,
    \begin{equation*}
      \mathrm{vec}(S^{(1)}) = (L_d\otimes \cdots \otimes L_1)
      \mathrm{vec}( T^{(1)})
      \iff
      \mathrm{vec}(\bb S) = (L_d\otimes \cdots \otimes L_1) \mathrm{vec}(\bb T).
    \end{equation*}
\end{proof}
Notice that formula~\eqref{eq:krontomu} is precisely the $d$-dimensional generalization of 
formula \eqref{eq:gen2dkron}. Hence,
tasks written as in formula~\eqref{eq:krontomu}
can be equivalently stated and computed more efficiently
again by formula \eqref{eq:tuckerop},
without assembling the large-sized matrix $L_d \otimes \cdots \otimes L_1$.

We can then summarize as follows: the \emph{element-wise} formulation \eqref{eq:nestedmumode},
the \emph{tensor} formulation \eqref{eq:tuckerop} and the \emph{vector} 
formulation \eqref{eq:krontomu} can all be used to compute the entries of the 
tensor $\bb S$. However, in light of the considerations for the $\mu$-mode 
product, only the tensor formulation can be efficiently computed
by $d$ calls of level 3 BLAS,
{\color{black} with
  an overall computational cost of $\mathcal{O}(n^{d+1})$ for the
  case $m_\mu=n_\mu=n$}. This is the reason why the relevant functions
of our package \kP{} 
{\color{black} are based on formulation~\eqref{eq:tuckerop}.
  \begin{remark}
  The implementation
  of a \emph{single} $\mu$-mode product in the function \verb+mump+
  of \kP{} involves two
  explicit permutations of the tensor
  (except the $1$-mode and the $d$-mode products,
  which are realized without explicitly permuting, thanks to the
  design of the function \verb+reshape+ in \matlab).
  On the other hand, the function \verb+tucker+, which realizes
  the Tucker operator~\eqref{eq:tuckerop},
  performs a composition of any pair of consecutive permutations,
  thus reducing their overall number.
  In fact, this is important when dealing with large-sized tensors, because
  the cost of permuting is not negligible
  due to the underlying alteration of the memory layout.
  For this reason, several algorithms which 
  further reduce or completely avoid permutations in an efficient way
  have been developed  (see, for
  instance~\cite{JBPSV15,R16,SB18,M18}).
  In this context, for instance, it is possible
  to use the function \verb+pagemtimes+
  to efficiently realize a ``Loops-over-GEMMs'' strategy.
  However, as this function has been recently introduced in
  MathWorks MATLAB\textsuperscript{\textregistered} R2020b and it is still
  not available in the latest stable GNU Octave release
  7.1.0, for compatibility reasons we do not follow this approach.
\end{remark}}%
Notice that the definition of $\mu$-mode product and its realization through the 
function \verb+mump+ can be easily extended to
the case in which instead of a matrix $L$ we have a
\emph{matrix-free} operator $\mathcal{L}$.
\begin{definition} \label{def:mma}
  Let $\mathcal{L}: \CC^{m_\mu} \to \CC^{n}$ be an operator. Then the
  \emph{$\mu$-mode action} of $\boldsymbol{T}$ with $\mathcal{L}$,
  still denoted $\bb S = \bb T \times_\mu \mathcal{L}$, is the tensor
  $\boldsymbol S\in\CC^{m_1\times\cdots\times m_{\mu-1}\times n\times m_{\mu+1}\times\cdots \times m_d}$
  obtained by the action of the operator $\mathcal{L}$ on the $\mu$-fibers 
  of $\boldsymbol{T}$.
\end{definition}
In \matlab{}, if the operator $\mathcal{L}$ is represented by the 
function \verb+Lfun+ which operates on columns, we can implement the 
$\mu$-mode action by
\begin{verbatim}

mumpfun = @(T,Lfun,mu) ipermute(reshape(Lfun(mumat(T,mu)),...
                       [n,m([1:mu-1,mu+1:d])]),[mu,1:mu-1,mu+1:d]);

\end{verbatim}
The corresponding generalization of the Tucker operator, denoted again by
\begin{equation}\label{eq:tuckerfun}
  \bb S = \bb T \times_1 \mathcal{L}_1 \times_2 \cdots \times_d \mathcal{L}_d
\end{equation}
{\color{black}and implemented
    in \kP{} in the function \verb+tuckerfun+},
follows straightforwardly. Clearly, in this case, some properties of the
Tucker operator \eqref{eq:tuckerop}, such as the
{\color{black}aforementioned} invariance with respect to the ordering
of the $\mu$-mode product operations, may not hold anymore for generic 
operators $\mathcal{L}_\mu$.
Generalization~\eqref{eq:tuckerfun} is useful in some instances,
see Remark~\ref{rem:interpfun}
and Section~\ref{sec:specdecex} for an example.
We remark that such an extension is not
available in some other popular tensor algebra toolboxes, such as
\emph{Tensor Toolbox for MATLAB} \cite{BKO21} --- which does not
have GNU Octave support, too --- and \emph{Tensorlab}
\cite{VDSVBDL16}, both of which
are more devoted to tensor decomposition and related topics.

  The  $\mu$-mode product is {\color{black}also useful for computing}
  the action of the Kronecker sum (see Appendix \ref{app:A} for its definition)
  of the $L_\mu$ matrices to a vector $\bb v$, that is
  \begin{equation}\label{eq:actkron}
    (L_d\oplus\cdots\oplus L_1)\bb v=\mathrm{vec}\left(\sum_{\mu=1}^{d}(\bb V \times_\mu L_\mu)\right),
  \end{equation}
  where $\bb v=\mathrm{vec}(\bb V)$. In fact, as it can be noticed
  from formula~\eqref{eq:mmp}, the identity matrix is the identity element of the
  $\mu$-mode product. Combining this observation with Lemma \ref{lem:krontomu}, 
  we easily obtain formula \eqref{eq:actkron}.
  In our package \kP, the matrix resulting from the Kronecker sum on the 
  left hand side of equality~\eqref{eq:actkron} can be computed as \verb+kronsum(L)+, 
  where
  \verb+L+ is the cell array containing $L_\mu$ in \verb+L{mu}+.
  On the other hand, its action on $\bb v$ can be computed equivalently
  in tensor formulation, without forming the matrix itself, by 
  \verb+kronsumv(V,L)+.

\section{Problems formulation in $d$ dimensions} \label{sec:1d2d}
In this section we discuss in more detail the problems 
that were briefly introduced in Section \ref{sec:intro}. Their 
generalization to arbitrary dimension $d$ is addressed thanks to
the common framework presented in Section \ref{sec:mmp}.

\subsection{Pseudospectral decomposition}\label{sec:specdec12d}
Suppose that a function $f\colon R\to\CC$, with $R\subseteq\RR$,
can be expanded into a series
\begin{equation*}
    f(x)=\sum_{i=1}^\infty f_i\phi_i(x),
\end{equation*}
where $f_i$ are complex scalar coefficients and $\phi_i(x)$ are 
complex functions orthonormal with respect to the 
standard $L^2(R)$ inner product, i.e.
\begin{equation*}
    \int_R\phi_i(x)\overline{\phi_j(x)}dx=\delta_{ij}{\color{black},} \quad \forall i,j.
\end{equation*}
Then, the spectral coefficients $f_i$ are defined by
\begin{equation*}
    f_i=\int_R f(x)\overline{\phi_i(x)}dx{\color{black},}
\end{equation*}
and can be approximated by a quadrature formula.
Usually, in this context, specific Gaussian quadrature formulas are employed, 
whose node and weights vary depending on the chosen family of basis functions.  
If we consider $q$
quadrature nodes $\xi^k$ and weights $w^k$, we
can compute the first $m$  \emph{pseudospectral}
  coefficients by
\begin{equation*}
  \hat f_i=\sum_{k=1}^q f(\xi^k)\overline{\phi_i(\xi^k)}w^k\approx
  f_i,\quad 1\le i\le m.
\end{equation*}
By collecting the values $\overline{\phi_i(\xi^k)}$
in position
$(i,k)$ of the matrix $\Psi\in\CC^{m\times q}$ and the values $f(\xi^k)w^k$
in the vector $\bb f_{\bb w}$, we can compute the pseudospectral
coefficients by
means of the single matrix-vector product
\begin{equation*}
    \hat {\bb f}=\Psi\bb f_{\bb w}.
\end{equation*}

In the two-dimensional case, the coefficients of a pseudospectral expansion in 
a tensor product basis (see, for instance, \cite[Ch.~6.10]{B00}) are given by
\begin{equation*}
    \hat f_{i_1 i_2}={\color{black}\sum_{k_2=1}^{q_2}\sum_{k_1=1}^{q_1}}
    f(\xi_1^{k_1},\xi_2^{k_2})\overline{\phi_{i_1}^1(\xi_1^{k_1})}
    \overline{\phi_{i_2}^2(\xi_2^{k_2})}w_1^{k_1}w_2^{k_2},
\end{equation*}
which can be efficiently computed as
\begin{equation*}
    \hat {\bb F}=\Psi_1 \bb F_{\bb W} \Psi_2^{\sf T},
\end{equation*}
where $\Psi_\mu\in\CC^{m_\mu\times q_\mu}$ has element $\overline{\phi_{i_\mu}^\mu(\xi_\mu^{k_\mu})}$ in
position $(i_\mu,k_\mu)$, with $\mu=1,2$, and $\bb F_{\bb W}$ is the
matrix with
element $f(\xi_1^{k_1},\xi_2^{k_2})w_1^{k_1}w_2^{k_2}$
in position $(k_1,k_2)$.

In general, the coefficients of a  $d$-dimensional pseudospectral expansion 
in a tensor product basis are given by
\begin{equation*}
  \hat f_{i_1 \ldots i_d}={\color{black}\sum_{k_d=1}^{q_d}\cdots\sum_{k_1=1}^{q_1}}
  f(\xi_1^{k_1},\ldots,\xi_d^{k_d})\overline{\phi_{i_1}^1(\xi_1^{k_1})}\cdots
  \overline{\phi_{i_d}^d(\xi_d^{k_d})}w_1^{k_1}\cdots w_d^{k_d}.
\end{equation*}
In tensor formulation, the coefficients can be computed as
(see formulas \eqref{eq:tuckerop} and \eqref{eq:nestedmumode})
\begin{equation*}
      \hat{\bb F} = \bb F_{\bb W} \times_1 \Psi_1 \times_2 \cdots \times_d \Psi_d,
\end{equation*}
where $\Psi_\mu$ is the transform matrix with
element $\overline{\phi_{i_\mu}^\mu(\xi_\mu^{k_\mu})}$
in position $(i_\mu,k_\mu)$,  and we collect in the
order-$d$ tensors $\hat{\bb F}$ and $\bb F_{\bb W}$ the values
$\hat f_{i_1 \ldots i_d}$ and
$f(\xi_1^{k_1},\ldots,\xi_d^{k_d})w_1^{k_1}\cdots w_d^{k_d}$,
respectively.
The corresponding pseudospectral approximation of $f(\bb x)$ is
\begin{equation}\label{eq:specnd}
  \hat{f}(\bb x)={\color{black}\sum_{i_d=1}^{m_d}\cdots\sum_{i_1=1}^{m_1}}
  \hat{f}_{i_1 \ldots i_d}
  \phi_{i_1}^1(x_1)\cdots\phi_{i_d}^d(x_d),
\end{equation}
where $\bb x = (x_1,\ldots,x_d)$.
An application to Hermite--Laguerre--Fourier function decomposition is given
in Section~\ref{sec:specdecex}.

\subsection{Function approximation} \label{sec:funcapprox12d}
Suppose we are given an approximation of a univariate function $f(x)$ in the form
\begin{equation}\label{eq:approxser}
\tilde f(x)=\sum_{i=1}^m c_i \phi_i(x)\approx f(x),
\end{equation}
where $c_i$ are scalar coefficients and
$\phi_i(x)$ are generic (basis)
functions. This is the case, for example, in the context of
function interpolation or pseudospectral expansions. We are interested in the
evaluation of formula \eqref{eq:approxser} at given points
$x^\ell$, with $1\leq\ell\leq n$. This can be easily realized in a single
matrix-vector product: indeed, if we collect the coefficients $c_i$ in the
vector $\bb c \in \CC^m$ and we form the matrix
$\Phi\in\CC^{n\times m}$ with element $\phi_i(x^\ell)$ in position $(\ell,i)$,
the sought evaluation is given by
\begin{equation*}
\tilde{\bb f}=\Phi\bb c,
\end{equation*}
being $\tilde{\bb f}\in \CC^{n}$ the vector containing the approximated
function at the given set of evaluation points.

The extension of formula \eqref{eq:approxser} to the tensor product
bivariate case 
is straightforward (see, for instance, \cite[Ch.~XVII]{dB01}).
Indeed, in this case the approximating function is given by
\begin{equation}\label{eq:interp2d}
\tilde f(x_1,x_2)={\color{black}\sum_{i_2=1}^{m_2}\sum_{i_1=1}^{m_1}}
c_{i_1 i_2}\phi_{i_1}^1(x_1)\phi_{i_2}^2(x_2)\approx f(x_1,x_2),
\end{equation}
where $c_{i_1 i_2}$ represent scalar coefficients and
$\phi_{i_\mu}^\mu(x_\mu)$ the (univariate) basis function, with
$1\leq i_\mu \leq m_\mu$ and $\mu=1,2$.
Then, given a Cartesian grid of points $(x_1^{\ell_1},x_2^{\ell_2})$,
with $1\leq \ell_\mu \leq n_\mu$,
the evaluation of approximation \eqref{eq:interp2d} can be
computed efficiently in matrix formulation by
\begin{equation*}
\tilde {\bb F} = \Phi_1 \bb C \Phi_2^{\sf T}.
\end{equation*}
Here we collected the function evaluations
$\tilde{f}(x_1^{\ell_1},x_2^{\ell_2})$ in the matrix
$\tilde {\bb F}$, we formed the matrices
$\Phi_{\mu}\in\CC^{n_\mu\times m_\mu}$ of element
$\phi^\mu_{i_\mu}(x_\mu^{\ell_\mu})$ in
position $(\ell_\mu,i_\mu)$, and we let $\bb C$ be the matrix
of element $c_{i_1 i_2}$ in position $(i_1,i_2)$.

In general, the approximation of a $d$-variate function $f$ with tensor product 
basis functions is given by
\begin{equation}\label{eq:interpnd}
        \tilde f(\bb x)={\color{black}\sum_{i_d=1}^{m_d}\cdots\sum_{i_1=1}^{m_1}}
        c_{i_1\ldots i_d}\phi_{i_1}^1(x_1)\cdots\phi_{i_d}^d(x_d)\approx
        f(\bb x),
\end{equation}
where $c_{i_1\ldots i_d}$ represent scalar
coefficients while $\phi_{i_\mu}^\mu(x_\mu)$ the (univariate) basis functions,
with $1\leq i_\mu \leq m_\mu$.
Then, given a Cartesian grid of points
$(x_1^{\ell_1},\ldots,x_d^{\ell_d})$, with $1\leq \ell_\mu \leq n_\mu$,
the evaluation of approximation \eqref{eq:interpnd} can be expressed in
tensor formulation as
\begin{equation}\label{eq:tensinterp}
        \tilde {\bb F} = \bb C \times_1 \Phi_1 \times_2 \cdots \times_d \Phi_d,
\end{equation}
see formulas \eqref{eq:tuckerop} and \eqref{eq:nestedmumode}. Here we denote
$\Phi_{\mu}$ the matrix with element
$\phi_{i_\mu}^\mu(x_\mu^{\ell_\mu})$ in
position $(\ell_\mu,i_\mu)$, and we collect in the order-$d$
tensors $\bb C$ and $\tilde{\bb F}$ the coefficients and the resulting function
approximation at the evaluation points, respectively. 
We present an application to barycentric multivariate interpolation in 
Section~\ref{sec:funcapproxex}.

\begin{remark}
    Clearly, formula \eqref{eq:tensinterp} can be employed to evaluate a 
    pseudospectral approximation \eqref{eq:specnd} at a generic Cartesian grid of points, by 
    properly defining the involved tensor $\bb C$ and matrices $\Phi_\mu$.
    In the context of direct and inverse spectral transforms, for example 
    for the effective numerical solution of differential equations
    (see \cite{CCEOZ22}),
    one could be interested in the evaluation of 
    pseudospectral decompositions at the same grid of quadrature points 
    $(\xi_1^{k_1},\ldots,\xi_d^{k_d})$ used to approximate the spectral coefficients.
    Under standard hypothesis,
    this can be done by applying formula~\eqref{eq:tensinterp} with
    matrices $\Phi_\mu = \Psi_\mu^*$, where the symbol $*$ denotes the conjugate 
    transpose. Without forming explicitly the matrices $\Phi_\mu$, the desired 
    evaluation can be computed using the matrices $\Psi_\mu$ by means of the
    \kP{} function \verb+cttucker+.
\end{remark}
\begin{remark}\label{rem:interpfun}
  Several functions which perform
  the whole one-dimensional procedure of approximating a function and 
  evaluating it on a set of points, given suitable inputs, are available.
  {\color{black}
  This is the case, for example in the interpolation context, of the \matlab{}
  built-in functions \verb+spline+, \verb+interp1+ (that
  performs different kind of
  one-dimensional interpolations),  and \verb+interpft+
  (which performs a resample of the input values by means of FFT techniques),
  or of the functions provided by the
  \textsc{QIBSH++}   library \cite{BM21} in the approximation context.}
  Yet, it is possible to extend the usage of this kind of 
  functions to the approximation in the $d$-dimensional
  tensor setting by means of
  concatenations of $\mu$-mode actions (see Definition \ref{def:mma}), 
  yielding the generalization of the Tucker operator~\eqref{eq:tuckerfun}.
  In practice, we can perform this task with the \kP{}
  function \verb+tuckerfun+, see the numerical example
  in Section~\ref{sec:specdecex}.
\end{remark}

\subsection{Action of the matrix exponential}\label{sec:exp12d}
Suppose we want to solve the linear Partial Differential Equation
(PDE)
\begin{equation}\label{eq:heat1D-cont}
\left\{
\begin{aligned}
  \partial_tu(t,x) &= \mathcal{A}u(t,x), \quad t>0,
  \quad x\in\Omega\subset\RR{\color{black},}\\
  u(0,x)&=u_0(x),
\end{aligned}\right.
\end{equation}
coupled with suitable boundary conditions, where $\mathcal{A}$ is a linear
time-independent spatial (integer or fractional) differential operator,
typically stiff.
The application of the method of lines to equation \eqref{eq:heat1D-cont}, by
discretizing first the spatial variable, e.g.~by finite differences or spectral
differentiation, leads to the system of Ordinary Differential Equations
(ODEs)
\begin{equation}\label{eq:heat1D-fd}
    \left\{
    \begin{aligned}
        \bu'(t) &= A\bu(t), \quad t>0{\color{black},}\\
        \bu(0)&=\bu_0{\color{black},}
    \end{aligned}\right.
\end{equation}
for the unknown vector $\bu(t)$.
Here, $A\in\CC^{n\times n}$
is the matrix which approximates the differential operator
$\mathcal{A}$ on the grid points $x^\ell$,
with $1\le \ell\le n$. The exact solution
of system~\eqref{eq:heat1D-fd} is obviously
$\bu(t)=\exp(tA)\bu_0$ and, if the size of $A$ allows, it can be effectively
computed by Pad\'e or Taylor approximations (see~\cite{AMH09,CZ19}).
If the size of $A$ is too large, then one has to rely on algorithms to
approximate  the action of the matrix exponential $\exp(tA)$
on the vector $\bu_0$. Examples of such methods are \cite{AMH11,NW12,GRT18,CCZ20}.

Suppose now we want to solve instead
\begin{equation}\label{eq:heat2D-cont}
    \left\{
    \begin{aligned}
      \partial_tu(t,x_1,x_2) &= \mathcal{A}u(t,x_1,x_2),
      \quad t>0, \quad (x_1,x_2)\in\Omega\subset\RR^2{\color{black},}\\
        u(0,x_1,x_2)&=u_0(x_1,x_2),
    \end{aligned}\right.
\end{equation}
coupled again with suitable boundary conditions.
If PDE \eqref{eq:heat2D-cont} admits a Kronecker 
structure, such as for some linear Advection--Diffusion--Absorption (ADA)
equations on tensor product domains or linear Schr\"odinger equations with
a potential in Kronecker form (see \cite{CCEOZ22}
for more details and
examples),
then the method of lines yields the system
of ODEs
\begin{equation}\label{eq:heat2D-fd}
    \left\{\begin{aligned}
    \bu'(t) &= \left(I_2\otimes A_1+A_2 \otimes I_1\right)
    \bu(t), \quad t>0{\color{black},} \\
        \bu(0)&=\bu_0.
    \end{aligned}\right.
\end{equation}
Here $A_\mu$, with $\mu=1,2$, represent the one-dimensional stencil matrices
corresponding to the discretization of the one-dimensional differential 
operators that constitute $\mathcal{A}$ on the grid points $x_\mu^{\ell_\mu}$,
with $1\le \ell_\mu\le n_\mu$. 
Moreover, the notation $I_\mu$ stands for identity
matrices of size $n_\mu$, and 
the component $\ell_1+(\ell_2-1)n_1$ of $\bu$
corresponds to the grid point
$(x_1^{\ell_1},x_2^{\ell_2})$, that is
\begin{equation*}
    u_{\ell_1+(\ell_2-1)n_1}(t)\approx u(t,x_1^{\ell_1},x_2^{\ell_2}).
\end{equation*}
This, in turn, is consistent with the linearization of the indexes of
{\color{black} the $\mathrm{vec}$ operator defined in Appendix \ref{app:A}}.

Clearly, the solution of system~\eqref{eq:heat2D-fd} is given by
\begin{equation}\label{eq:2D-discrete}
    \bu(t) = \exp\left(t(I_2\otimes A_1+A_2 \otimes I_1)\right)\bu_0,
\end{equation}
which again could be computed by any method to compute the action of the matrix
exponential on a vector. Remark that, since the matrices
$I_2\otimes A_1$ and $A_2\otimes I_1$ commute and using the properties of the
Kronecker product (see Appendix \ref{app:A}), one could write everything in terms of
the exponentials of the \textit{small-sized} matrices $A_\mu$. Indeed, we have
\begin{equation*}
    \begin{split}
      \bu(t) &= \exp\left(t(I_2\otimes A_1+A_2 \otimes I_1)\right)\bu_0 =
          \exp(t(I_2\otimes A_1))\exp(t(A_2 \otimes I_1))\bu_0  \\
          &= \left(I_2 \otimes \exp(tA_1)\right)\left(\exp(tA_2)\otimes I_1\right)\bu_0
          = (\exp(tA_2)\otimes\exp(tA_1))\bu_0.
    \end{split}
\end{equation*}
However, as in general the matrices $\exp(tA_\mu)$ are full, their Kronecker 
product results in a large and full matrix to be multiplied into $\bu_0$, 
which is an extremely inefficient approach. Nevertheless, if we fully exploit 
the tensor structure of the problem, we can still compute the solution of the 
system efficiently just in terms of the exponentials $\exp(tA_\mu)$.
Indeed, let $\bU(t)$ be the $n_1\times n_2$ matrix whose stacked columns form 
the vector $\bu(t)$, that is
\begin{equation*}
\mathrm{vec}(\bU(t))=\bu(t).
\end{equation*}
Then, using this matrix notation and by means of the properties of the
Kronecker product, problem \eqref{eq:heat2D-fd} takes the form
\begin{equation*}
    \left\{\begin{aligned}
        \bU'(t) &= A_1\bU(t) + \bU(t) A_2^{\sf T}, \quad t>0{\color{black},}\\
        \bU(0)&= \bU_0,
    \end{aligned}\right.
\end{equation*}
and it is well-known (see~\cite{N69}) that its solution can be computed
in matrix formulation as
\begin{equation*}
    \bU(t) =  \exp(t A_1)\bU_0\exp(t A_2)^{\sf T}.
\end{equation*}

In general, the $d$-dimensional version of solution~\eqref{eq:2D-discrete} is
\begin{equation*}
  \bu(t) =
  \exp\left(t\sum_{\mu=1}^d \left(I_{d}\otimes \cdots \otimes I_{\mu+1}\otimes
  A_{\mu}\otimes I_{\mu-1}\otimes \cdots \otimes I_1\right)\right)\bu_0,
\end{equation*}
which can be written in more compact notation as
\begin{equation} \label{eq:expmat}
    \bu(t)=\exp\left(t\left(A_d\oplus \cdots \oplus A_1\right)\right)\bu_0.
\end{equation}
Here, $A_\mu$ are square matrices of size $n_\mu$, and
$\bu_0$ is a vector of length $N=n_1\cdots n_{d}$. Then, similarly
to the two-dimensional case, we have
\begin{equation*}
  \bb u(t)=\exp\left(t(A_d\oplus \cdots \oplus A_1)\right)\bu_0
=(\exp(tA_d)\otimes\cdots\otimes\exp(tA_1))\bu_0.
\end{equation*}
Finally, using Lemma \ref{lem:krontomu}, we have
\begin{equation}\label{eq:exptucker}
  \bb U(t) = \bb U_0 \times_1 \exp(tA_1) \times_2 \cdots \times_d \exp(tA_d),
\end{equation}
where $\bb U(t)$ and $\bb U_0$ are $d$-dimensional tensors such that 
$\bb u(t) = \mathrm{vec}(\bb U(t))$ and $\bb u_0 = \mathrm{vec}(\bb U_0)$. 
Hence, the action of the large-sized matrix exponential appearing in formula
\eqref{eq:expmat} can be computed by the Tucker operator \eqref{eq:exptucker}
which just involves the small-sized matrix exponentials $\exp(t A_\mu)$.
For an application in the context of solution of an 
ADA linear evolutionary equation with spatially variable coefficients,
see Section~\ref{sec:expex}.

\subsection{Preconditioning of linear systems}\label{sec:prec12d}
Suppose we want to solve the semilinear PDE
\begin{equation}\label{eq:semilin1D-cont}
    \left\{
    \begin{aligned}
        \partial_tu(t,x) &= \mathcal{A}u(t,x) + f(t,u(t,x)), 
        \quad t>0, \quad x\in\Omega\subset\RR{\color{black},}\\
        u(0,x)&=u_0(x),
    \end{aligned}\right.
\end{equation}
coupled with suitable boundary conditions, where $\mathcal{A}$ is a linear 
time-independent spatial differential operator and $f$ is a nonlinear 
function. 
Using the method of lines, similarly to what led to 
system \eqref{eq:heat1D-fd}, we obtain
\begin{equation}\label{eq:semilin1D-disc}
    \left\{\begin{aligned}
        \bu'(t) &=A\bu(t)+\bb f(t,\bu(t)), \quad t>0{\color{black},} \\
        \bu(0)&=\bu_0.
    \end{aligned}\right.
\end{equation}
A common approach to {\color{black}integrating system \eqref{eq:semilin1D-disc}}
in time 
involves the use of IMplicit EXplicit (IMEX) schemes. For instance, the
application of the well-known backward-forward Euler method with
constant time step $\tau$ leads to the solution of the linear system
\begin{equation*}
M\bb u_{k+1}= \bb u_k + \tau \bb f (t_k,\bb u_k)
\end{equation*}
at every time step, where $M=(I-\tau A)\in\CC^{n\times n}$
and $I$ is an identity matrix of suitable
size. If the space discretization allows (second order centered finite
differences, for example), the system can then be solved by means of the
very efficient Thomas algorithm. If, on the other hand, this is not the case,
a suitable direct or (preconditioned) iterative method can be employed.

Let us consider now the two-dimensional version of the semilinear
PDE~\eqref{eq:semilin1D-cont}, i.e.
\begin{equation}\label{eq:semilin2D-cont}
    \left\{
    \begin{aligned}
        \partial_tu(t,x_1,x_2) &= \mathcal{A}u(t,x_1,x_2) + f(t,u(t,x_1,x_2)), 
        \quad t>0, \quad (x_1,x_2)\in\Omega\subset\RR^2,\\
        u(0,x_1,x_2)&=u_0(x_1,x_2), 
    \end{aligned}\right.
\end{equation}
again with suitable boundary conditions, $\mathcal{A}$ a linear 
time-independent spatial differential operator and $f$ a nonlinear function. 
As for
equation~\eqref{eq:heat2D-cont}, if the PDE admits a Kronecker sum
structure, the 
application of the method of lines leads to 
\begin{equation} \label{eq:semilin2D-discvec}
\left\{\begin{aligned}
\bu'(t)&=(I_2\otimes A_1 + A_2\otimes I_1)\bu(t)+\bb f(t,\bu(t)),
\quad t>0{\color{black},}\\
    \bu(0)&=\bu_0,
\end{aligned}\right.
\end{equation}
which can be integrated in time again
by means of the backward-forward Euler method.
{\color{black}The matrix of the resulting linear system to
be solved at every time step is now}
\begin{equation*}
  M=I_2\otimes M_1+M_2\otimes I_1=
  I_2\otimes \left(\frac{1}{2}I_1-\tau A_1\right)+
  \left(\frac{1}{2}I_2-\tau A_2\right)\otimes I_1.
\end{equation*}
If we use an iterative method, we
{\color{black}
can obtain the action of the matrix $M$ to a vector $\bb v$ as
\begin{equation*}
  M_1\bb V+\bb VM_2^{\sf T} = \bb V_{\!M},\quad
  \mathrm{vec}(\bb V)=\bb v,
\end{equation*}
by observing that
\begin{equation*}
M \bb v=\mathrm{vec}(\bb V_{\!M}).
\end{equation*}}%

Moreover, examples of effective preconditioners for this kind of linear 
systems are the ones of Alternating Direction Implicit
(ADI) type (see \cite{AR20}).
In this case, {\color{black}we can use the product of the matrices 
  arising from the discretization of equation~\eqref{eq:semilin2D-cont} after
  neglecting all the spatial variables but one in the operator $\mathcal{A}$}.
We obtain then the preconditioner
\begin{equation}\label{eq:prec}
  (I_2-\tau A_2)\otimes(I_1-\tau A_1) = P_2\otimes P_1 = P,
\end{equation}
which is expected to be effective since $P=M+\mathcal{O}(\tau^2)$.
In addition, the action of $P^{-1}$ to a vector $\bb v$ can be
efficiently obtained as
\begin{equation*}
P_1^{-1}\bb VP_2^{-\sf{T}} = \bb V_{\!P^{-1}},
\end{equation*}
by noticing that
\begin{equation*}
  P^{-1}\bb v=(P_2^{-1}\otimes P_1^{-1})\bb v=\mathrm{vec}(\bb V_{\!P^{-1}}).
  \end{equation*}
\begin{remark}
  Another approach of solution to equation \eqref{eq:semilin2D-discvec} would
  be to write the equivalent matrix formulation of the problem, i.e.
\begin{equation*}
\left\{
\begin{aligned}
\bU'(t)&=A_1 \bU(t)+\bU(t)A_2^{\sf T}+\bb F(t,\bU(t)), \quad t>0{\color{black},}\\
\bU(0)&=\bU_0,
\end{aligned}\right.
\end{equation*}
and then apply appropriate algorithms to integrate it numerically,
mainly based on the solution of Sylvester equations.
This is the approach pursued, for example, in \cite{KS20}.
\end{remark}

In general, for a $d$-dimensional semilinear problem with a Kronecker sum
structure,
the linear system to be solved at every time step {\color{black}
  has now matrix}
\begin{equation*}
M = M_d \oplus \cdots \oplus M_1, \quad 
M_\mu=\left(\frac{1}{d}I_\mu-\tau A_\mu\right).
\end{equation*}
Again, the action of the matrix $M$ on a vector $\bb v$ can be computed without
assembling the matrix (see equivalence \eqref{eq:actkron}).
Finally, an effective preconditioner for the
linear system is a straightforward generalization of
formula~\eqref{eq:prec}, i.e. 
\begin{equation*}
  (I_d-\tau A_d)\otimes\cdots\otimes(I_1-\tau A_1)=
  P_d\otimes\cdots\otimes P_1 = P.
\end{equation*}
Similarly to the two-dimensional case, its inverse action to a vector
$\bb v$ can be computed efficiently  as
\begin{equation}\label{eq:preconnd}
\bb V \times_1  P_1^{-1} \times_2 \cdots \times_d P_d^{-1} = \bb V_{\!P^{-1}},
\end{equation}
see Lemma \ref{lem:krontomu}.
{\color{black}In our package \kP{}, formula \eqref{eq:preconnd} can be realized
  without explicitly inverting the matrices $P_\mu$ by using the
  function \verb+itucker+. We notice that this is another feature
  not available in
  the tensor algebra toolboxes mentioned in Section~\ref{sec:mmp}.}
For an example of application of these techniques, in the context of solution
of evolutionary diffusion--reaction equations, see Section~\ref{sec:precex}.

We finally notice that there exist also specific techniques to solve
linear systems in
Kronecker form, usually arising in the discretization of time-independent
differential equation, see for instance \cite{PS16,CK20}.

\section{{\color{black}Numerical experiments}} \label{sec:code}
We present in this section some numerical experiments of
the proposed $\mu$-mode approach for tensor-structured problems,
which make extensively use of the functions contained in our package \kP{}.
{\color{black}We remark that, when we employ Cartesian grids of points,
they  have been
produced by the \matlab{} command \verb+ndgrid+.}
If instead one would prefer
to use the ordering induced by the \verb+meshgrid+ command (which,
however, works only up to dimension three), it is enough to
interchange the first and the second matrix in the Tucker
operator~\eqref{eq:tuckerop}.
The resulting tensor is then the $(2,1,3)$-permutation of $\bb S$
in Definition~\ref{def:tucker}.

All the numerical experiments have been performed
with MathWorks MATLAB\textsuperscript{\textregistered} R2019a on an
Intel\textsuperscript{\textregistered} Core\textsuperscript{\texttrademark}
i7-8750H CPU with 16 GB of RAM. The degrees of freedom
  of the problems have been kept at a moderate size, in order
  to be reproducible with the package \kP{} in a few seconds on a
  personal laptop.
{\color{black}
\subsection{Code validation}
In this section we validate the \verb+tucker+ function of \kP, by
comparing it to the corresponding functions of the toolboxes
mentioned in Section~\ref{sec:mmp},
i.e.~\verb+ttm+ and \verb+tmprod+ of Tensor Toolbox for MATLAB
and Tensorlab, respectively. We performed several tests on tensors
of different orders and sizes and the three functions always produced the
same output (up to round-off unit) at comparable computational times.
For simplicity of exposition, we report in Figure~\ref{fig:codeval}
just the wall-clock times of the experiments with tensors of order
$d=3$ and $d=6$. For each selected value of $d$, we take
as tensors and matrices sizes $m_\mu=n_\mu=n$,
$\mu=1,\ldots,d$, for different values of $n$,
in such a way that the number of degrees of freedom $n^d$ ranges
from $N_\mathrm{min}=12^6$ to $N_\mathrm{max}=18^6$. The input tensors and
matrices have normal distributed random values, and the complete
code can be found in the script \verb+code_validation.m+.}
\begin{figure}[!ht]
\centering
%
%
\definecolor{myblue}{rgb}{0.00000,0.44700,0.74100}%
\begin{tikzpicture}

\begin{axis}[%
width=1.7in,
height=2.4in,
at={(0.759in,0.481in)},
scale only axis,
xmin=100,
xmax=370,
xtick={144, 196, 256, 324},
xlabel style={font=\color{white!15!black}},
xlabel style={font=\footnotesize},
xticklabel style={font=\footnotesize},
xlabel={$n$},
ymode=log,
ymin=0.01,
ymax=1,
yminorticks=true,
ylabel style={font=\color{white!15!black}},
ylabel style={font=\footnotesize},
yticklabel style={font=\footnotesize},
ylabel={Wall-clock time (s)},
axis background/.style={fill=white}
  ]
  \addplot [thick, color=myblue, mark=x, 
    mark options={solid, myblue}, forget plot]
  table[row sep=crcr]{%
144	3.19e-2\\
196	1.05e-1\\
256	2.77e-1\\
324	6.06e-1\\
};

\addplot [thick, color=red, mark=o, 
  mark options={solid, red}, forget plot]
  table[row sep=crcr]{%
144	2.44e-2\\
196	9.49e-2\\
256	2.63e-1\\
324	5.81e-1\\
};

\addplot [thick, color=green, mark=diamond, 
  mark options={solid, green}, forget plot]
  table[row sep=crcr]{%
144	2.19e-2\\
196	8.36e-2\\
256	2.24e-1\\
324	5.21e-1\\
};

\end{axis}

\begin{axis}[%
width=1.7in,
height=2.4in,
at={(3in,0.481in)},
scale only axis,
xmin=10,
xmax=20,
xtick={12, 14, 16, 18},
xlabel style={font=\color{white!15!black}},
xlabel style={font=\footnotesize},
xticklabel style={font=\footnotesize},
xlabel={$n$},
ymode=log,
ymin=0.01,
ymax=1,
yminorticks=true,
ylabel style={font=\color{white!15!black}},
ylabel style={font=\footnotesize},
yticklabel style={font=\footnotesize},
axis background/.style={fill=white},
legend style={at={(0.95,0.05)}, anchor=south east, legend cell align=left, align=left, draw=white!15!black},
legend style = {font=\footnotesize}
  ]
  \addplot [thick, color=myblue, mark=x,
     mark options={solid, myblue}]
  table[row sep=crcr]{%
12	5.56e-2\\
14	1.73e-1\\
16	3.84e-1\\
18	7.26e-1\\
  };
\addlegendentry{\texttt{ttm}}

\addplot [thick, color=red, mark=o, 
    mark options={solid, red}]
  table[row sep=crcr]{%
12	5.49e-2\\
14	1.65e-1\\
16	5.18e-1\\
18	7.04e-1\\
};
\addlegendentry{\texttt{tmprod}}

\addplot [thick, color=green, mark=diamond, 
    mark options={solid, green}]
  table[row sep=crcr]{%
12	3.48e-2\\
14	1.12e-1\\
16	3.05e-1\\
18	5.25e-1\\
};
\addlegendentry{\texttt{tucker}}

\end{axis}

\end{tikzpicture}%
  \caption{\color{black}Wall-clock times for different realizations of the
    Tucker operator~\eqref{eq:tuckerop}
    with the functions \texttt{ttm}, \texttt{tmprod}, and \texttt{tucker}.
    The left plot refers to the case $d=3$, while the right plot
    refers to the case $d=6$.
    Each test has been repeated several times in order to avoid
    fluctuations.}
  \label{fig:codeval}
\end{figure}
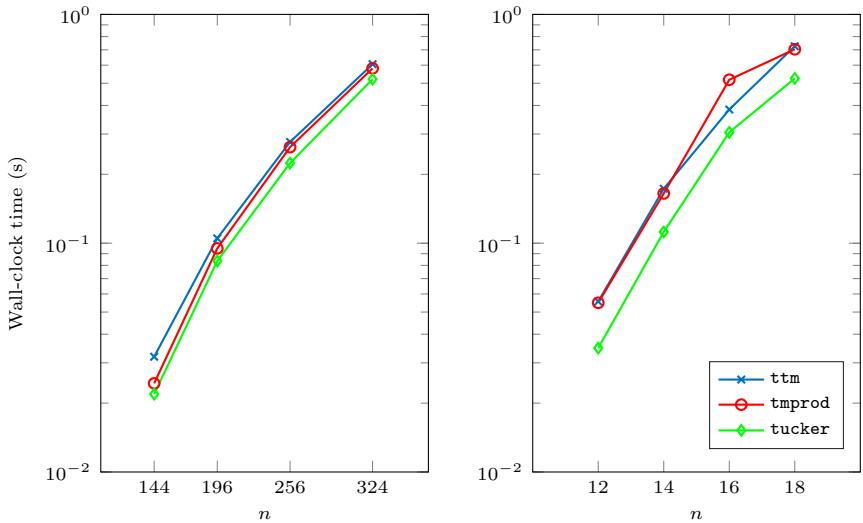
\subsection{Hermite--Laguerre--Fourier function
  decomposition}\label{sec:specdecex}
We are interested in the approximation, by means of a
pseudospectral decomposition,
of the trivariate function
\begin{equation*}
  f(\bx)=\frac{x_2^2\sin(20x_1)\sin(10x_2)\exp(-x_1^2-2x_2)}{\sin(2\pi x_3)+2},
  \quad \bx=(x_1,x_2,x_3)\in\Omega,
\end{equation*}
where $\Omega=[-b_1,b_1]\times[0,b_2]\times[a_3,b_3]$.
{\color{black}The decays in the first and second directions
  and the periodicity
  in the third direction
  suggest the use of a} Hermite--Laguerre--Fourier (HLF) expansion.
This mixed
transform is useful, for instance, for the solution of differential
equations with cylindrical coordinates by spectral methods,
see~\cite{BLS09}.
We then introduce the
\emph{normalized and scaled} Hermite functions
(orthonormal in $L^2(\RR)$)
\begin{equation*}
    \mathcal{H}^{\beta_1}_{i_1}(x_1)=\sqrt{\frac{\beta_1}{\sqrt{\pi}2^{i_1-1}
      (i_1-1)!}}H_{i_1}(\beta_1 x_1)\rme^{-\beta_1^2x_1^2/2}{\color{black},}
\end{equation*}
where $H_{i_1}$ is the (physicist's) Hermite polynomial of degree $i_1-1$.
We consider the $m_1$ scaled Gauss--Hermite quadrature points
$\{\xi_1^{k_1}\}_{k_1}$ and
define $\Psi_1\in\RR^{m_1\times m_1}$
to be the corresponding transform matrix with element
$\mathcal{H}^{\beta_1}_{i_1}(\xi^{k_1}_1)$ in position $(i_1,k_1)$.
The parameter
$\beta_1$ is chosen so that the quadrature points are contained
in $[-b_1,b_1]$ (see~\cite{T93}). This is possible by estimating the largest
quadrature point for the unscaled functions
by $\sqrt{2m_1+1}$ (see~\cite[Ch.~6]{S75}) and setting
\begin{equation*}
    \beta_1=\frac{\sqrt{2m_1+1}}{b_1}.
  \end{equation*}
Moreover, we consider the \emph{normalized and scaled} generalized Laguerre
functions (orthonormal in $L^2(\RR^+)$)
\begin{equation*}
  \mathcal{L}^{\alpha,\beta_2}_{i_2}(x_2)=
  \sqrt{\frac{\beta_2(i_2-1)!}{\Gamma(i_2+\alpha)}}
    L_{i_2}^{\alpha}(\beta_2 x_2)(\beta_2 x_2)^{\alpha/2}\rme^{-\beta_2 x_2/2}{\color{black},}
\end{equation*}
where $L_{i_2}^{\alpha}$ is the generalized Laguerre polynomial
of degree $i_2-1$.
We define $\Psi_2$  to be the corresponding transform matrix with element
$\mathcal{L}^{\alpha,\beta_2}_{i_2}(\xi^{k_2}_2)$ in position
$(i_2,k_2)$,
where $\{\xi_2^{k_2}\}_{k_2}$ are the $m_2$ scaled generalized
Gauss--Laguerre
quadrature points. The parameter $\beta_2$ is chosen, similarly to the
Hermite case, as
\begin{equation*}
    \beta_2=\frac{4m_2+2\alpha+2}{b_2},
  \end{equation*}
see \cite[Ch.~6]{S75} for the asymptotic estimate
{\color{black}which holds for $\lvert\alpha\rvert\ge 1/4$ and $\alpha>-1$}.
Finally, for the Fourier
decomposition, we obviously do not construct the transform matrix, but
we rely on a Fast Fourier Transform (FFT) implementation provided by the 
\matlab{} function \verb+interpft+, which performs a resample of the given
input values by means of FFT techniques.
We measure the approximation error,
for varying values of $n_\mu$, $\mu=1,2,3$,
by evaluating the pseudospectral decomposition
at a Cartesian grid of points $(x_1^{\ell_1},x_2^{\ell_2},x_3^{\ell_3})$,
with $1\le \ell_\mu\le n_\mu$. In
order to do that, we construct the matrices $\Phi_1$ and $\Phi_2$
containing the values of the Hermite and generalized Laguerre functions
at the points $\{x_1^{\ell_1}\}_{\ell_1}$ and $\{x_2^{\ell_2}\}_{\ell_2}$,
respectively. The relevant code for the approximation of $f$
and its evaluation, by using the \kP{} function \verb+tuckerfun+,
can be written as
\begin{verbatim}

PSIFUN{1} = @(f) PSI{1}*f;
PSIFUN{2} = @(f) PSI{2}*f;
PSIFUN{3} = @(f) f;
Fhat = tuckerfun(FW,PSIFUN);
PHIFUN{1} = @(f) PHI{1}*f;
PHIFUN{2} = @(f) PHI{2}*f;
PHIFUN{3} = @(f) interpft(f,n(3));
Ftilde = tuckerfun(Fhat,PHIFUN);

\end{verbatim}
where \verb+FW+ is the three-dimensional array containing
the values $f(\xi_1^{k_1},\xi_2^{k_2},\xi_3^{k_3})w_1^{k_1}w_2^{k_2}$,
{\color{black}where $\{\xi_3^{k_3}\}_{k_3}$ are the $m_3$ equispaced
  Fourier quadrature
points in $[a_3,b_3)$ and}
$\{w_\mu^{k_\mu}\}_{k_\mu}$, with $\mu=1,2$, are the scaled weights
of the Gauss--Hermite and generalized Gauss--Laguerre quadrature
rules, respectively. The values
$\{\xi_\mu^{k_\mu}\}_{k_\mu}$ and $\{w_\mu^{k_\mu}\}_{k_\mu}$, for $\mu=1,2$,
have been computed by the relevant Chebfun functions \cite{DHT14}.
The complete example can be found in the script 
\verb+example_spectral.m+, 

Given a prescribed accuracy, we look for the smallest number
of basis functions $(m_1,m_2,m_3)$ {\color{black}that achieve} it, and we measure 
the computational time needed to perform the approximation of $f$ and its evaluation
with the HLF method. As a term of comparison,
we consider the same experiment with a three-dimensional Fourier spectral 
approximation (FFF method): in fact,
for the size of the computational domain and the exponential decays
along the first and second
  {\color{black}directions}
of the function $f$  we are considering, it appears reasonable to
approximate $f$ by a periodic function in $\Omega$ and take advantage
of the efficiency of a three-dimensional FFT.

The results with $\alpha=4$,
$b_1=4$, $b_2=11$, $b_3=-a_3=1$, and $n_1=n_2=n_3=301$
evaluation points uniformly distributed in $\Omega$ are displayed in
Figure~\ref{fig:spectral}. 
As we can observe, the total number of degrees of freedom
needed by the HLF approach is always smaller than the
corresponding FFF one. In particular,
  despite the exponential decay along the second direction, the
  FFF method requires a very large number of Fourier coefficients
  along that direction in order to reach the most stringent accuracies. In
these situations,
the HLF method implemented with the $\mu$-mode approach
is preferable in terms of computational time to the well-established
implementation
by the FFT technique of the FFF method.

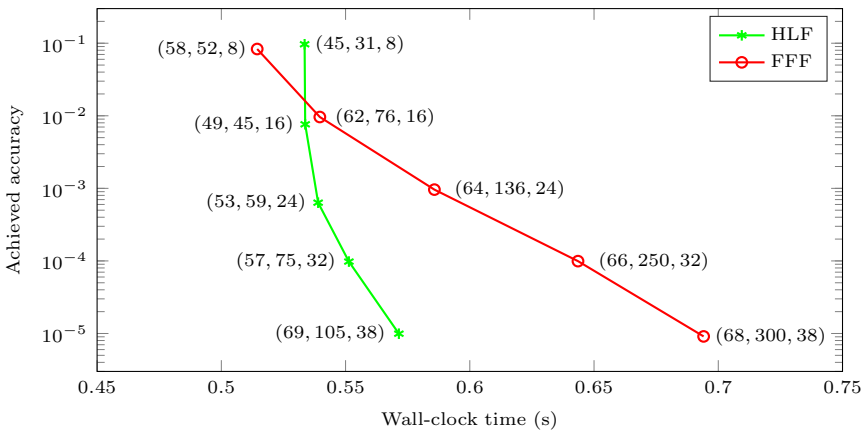
\begin{figure}[!ht]
  \centering
%
%
\begin{tikzpicture}

\begin{axis}[%
width=3.868in,
height=1.903in,
at={(0.78in,0.493in)},
scale only axis,
xmin=0.45,
xmax=0.75,
xlabel style={font=\color{white!15!black}},
xlabel={Wall-clock time (s)},
xlabel style={font=\footnotesize},
xticklabel style={font=\footnotesize},
ymode=log,
ymin=3e-06,
ymax=3e-1,
yminorticks=true,
ytick={1e-6,1e-5,1e-4,1e-3,1e-2,1e-1,1e0},
ylabel style={font=\color{white!15!black}},
ylabel={Achieved accuracy},
ylabel style={font=\footnotesize},
yticklabel style={font=\footnotesize},
axis background/.style={fill=white},
legend style={legend cell align=left, align=left, draw=white!15!black},
legend style = {font=\footnotesize}
]
\addplot [thick, color=green, mark=asterisk, mark options={solid, green}]
  table[row sep=crcr]{%
0.533518	0.0967219390298781\\
0.533745	0.00763343848248194\\
0.538953	0.000635001492530701\\
0.551342	9.79592289618978e-05\\
0.571458	9.93544678172208e-06\\
};
\addlegendentry{HLF}

\addplot [thick, color=red, mark=o, mark options={solid, red}]
  table[row sep=crcr]{%
0.514431	0.0831966928884739\\
0.539647	0.0096181125428191\\
0.585751	0.000960792365190478\\
0.643543	9.92904881693259e-05\\
0.694106	9.11375720188854e-06\\
};
\addlegendentry{FFF}

\node[left, align=left, font = \footnotesize]
at (axis cs:0.514,8.32e-02) {$(58,52,8)$};
\node[right, align=left, font = \footnotesize]
at (axis cs:0.534,9.67e-02) {$(45,31,8)$};

\node[left, align=left, font = \footnotesize]
at (axis cs:0.533,7.63e-03) {$(49,45,16)$};
\node[right, align=left, font = \footnotesize]
at (axis cs:0.542,9.62e-03) {$(62,76,16)$};

\node[left, align=left, font = \footnotesize]
at (axis cs:0.538,6.35e-04) {$(53,59,24)$};
\node[right, align=left, font = \footnotesize]
at (axis cs:0.59,9.61e-04) {$(64,136,24)$};

\node[left, align=left, font = \footnotesize]
at (axis cs:0.55,9.80e-05) {$(57,75,32)$};
\node[right, align=left, font = \footnotesize]
at (axis cs:0.648,9.93e-05) {$(66,250,32)$};

\node[left, align=left, font = \footnotesize]
at (axis cs:0.57,9.94e-06) {$(69,105,38)$};
\node[right, align=left, font = \footnotesize]
at (axis cs:0.695,9.11e-06) {$(68,300,38)$};
\end{axis}
\end{tikzpicture}%
  \caption{Achieved accuracies versus wall-clock times (in seconds) for the
    Hermite--Laguerre--Fourier (HLF) and the Fourier--Fourier--Fourier (FFF)
    approaches. The label of the marks in the plot indicates the number of basis
  functions used in each direction.}
  \label{fig:spectral}
\end{figure}

\subsection{Multivariate interpolation}\label{sec:funcapproxex}
Let us consider the approximation of a function $f(\bb x)$ through a
{\color{black}five-variate
interpolating polynomial in Lagrange form
\begin{equation}\label{eq:interp3d}
  p(\bx) = {\color{black}\sum_{i_5=1}^{m_5}\cdots\sum_{i_1=1}^{m_1}}
  f_{i_1\ldots i_5}L_{i_1}(x_1)\cdots L_{i_5}(x_5).
\end{equation}}%
Here $L_{i_\mu}(x_\mu)$ is the Lagrange polynomial of degree $m_\mu-1$
on a set $\{\xi^{k_\mu}_\mu\}_{k_\mu}$ of $m_\mu$ interpolation points written
in the second barycentric form, with $\mu=1,\ldots,5$, i.e.
\begin{equation*}
  L_{i_\mu}(x_\mu)=\frac{\frac{w^{i_\mu}_\mu}{x_\mu-\xi^{i_\mu}_\mu}}
  {\sum_{k_\mu}\frac{w^{k_\mu}_\mu}{x_\mu-\xi^{k_\mu}_\mu}},
  \quad
  w_\mu^{i_\mu} = \frac{1}{\prod_{k_\mu\neq i_\mu} (\xi_\mu^{i_\mu}-\xi_\mu^{k_\mu})},
\end{equation*}
while
{\color{black}$f_{i_1\ldots i_5} = f(\xi_1^{i_1},\ldots,\xi_5^{i_5})$.

  For our numerical example, we consider the {\color{black}five}-dimensional
  Runge function
\begin{equation*}
f(x_1,\ldots,x_5)=\frac{1}{1+16\sum_\mu x_\mu^2}
\end{equation*}
in the domain $[-1,1]^5$.} We choose as interpolation points
{\color{black}a Cartesian grid of} Chebyshev nodes
\begin{equation*}
  \xi_\mu^{k_\mu}=\cos\left(\frac{(2k_\mu-1)\pi}{2m_\mu}\right),\quad
  k_\mu=1,\ldots,m_\mu,
\end{equation*}
whose barycentric weights are
\begin{equation*}
  w_\mu^{k_\mu}=(-1)^{k_\mu+1}\sin\left(\frac{(2k_\mu-1)\pi}{2m_\mu}\right),\quad
  k_\mu=1,\ldots,m_\mu.
\end{equation*}
This is the {\color{black}five}-dimensional version of one of the examples
presented in
\cite[Sec.~6]{BT04}.
We evaluate the polynomial at a uniformly spaced Cartesian grid of points 
$(x_1^{\ell_1},\ldots,x_5^{\ell_5})$,
with $1\leq\ell_\mu\leq n_\mu$. Then,
approximation \eqref{eq:interp3d} at the just mentioned grid can be
computed
as
\begin{equation}\label{eq:laginterp}
\bb P = \bb F \times_1 L_1 \times_2 \cdots \times_5 L_5,
\end{equation}
where we collected the function evaluations at the interpolation points in the 
tensor $\bb F$ and $L_\mu$ contains the element $L_{i_\mu}(x_\mu^{\ell_\mu})$
in position $(\ell_\mu,i_\mu)$.
If we store the matrices $L_\mu$ in a cell \verb+L+, the corresponding
\matlab{} command for computing the desired approximation is
\begin{verbatim}

P = tucker(F,L);

\end{verbatim}
The results, for a number of evaluation points fixed to $n_\mu=n=35$ and 
varying number of interpolation points $m_\mu = m$, are reported in Figure \ref{fig:lagrange},
and the complete code can be found in the script \verb+example_interpolation.m+.
\begin{figure}[!ht]
\centering
%
%
\definecolor{mycolor1}{rgb}{0.00000,0.44700,0.74100}%
\begin{tikzpicture}

\begin{axis}[%
width=3.868in,
height=1.903in,
at={(0.766in,0.486in)},
scale only axis,
xmin=0,
xmax=50,
xtick={5,15,25,35,45},
xticklabels={{5},{15},{25},{35},{45}},
xlabel style={font=\footnotesize},
xticklabel style={font=\footnotesize},
xlabel={$m$},
ymode=log,
ymin=1e-6,
ymax=1,
ytick={1e-6,1e-4,1e-2,1e0},
yminorticks=true,
ylabel={Relative error},
ylabel style={font=\footnotesize},
yticklabel style={font=\footnotesize},
axis background/.style={fill=white},
legend style={legend cell align=left, align=left, draw=white!15!black},
legend style = {font=\footnotesize}
]
\addplot [ultra thick, color=mycolor1, only marks,
mark size = 2.5pt, mark=o, mark options={solid, mycolor1}]
  table[row sep=crcr]{%
5	0.330\\
15	0.0243\\
25	0.00189\\
35	0.000155\\
45	1.39e-05\\
};
\addlegendentry{Relative error}

\addplot [dashed, color=black]
  table[row sep=crcr]{%
5	0.27607\\
45	1.387e-5\\
};
\addlegendentry{Theoretical decay}
\end{axis}
\end{tikzpicture}%
\caption{Results for approximation \eqref{eq:laginterp} with an increasing
  number {\color{black}$m_\mu=m$}
  of interpolation points. 
  The relative error (blue circles) is computed 
         in maximum norm at the evaluation points. For reference, a dashed line 
         representing the theoretical decay estimate is added.}
\label{fig:lagrange}
\end{figure}
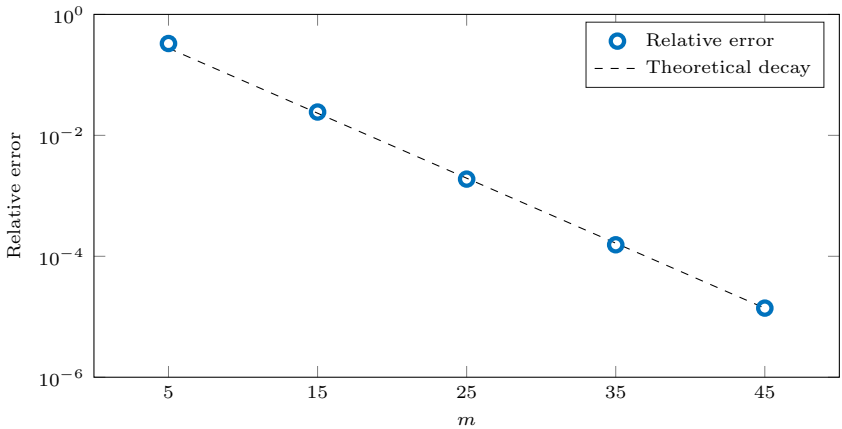

As expected, the error decreases according to the estimate
\begin{equation*}
\lVert f(\bb x) - p(\bb x) \lVert_\infty \approx K ^ {-m}, 
 \quad K = \frac{1}{4} + \sqrt{\frac{17}{16}},
\end{equation*}
see~\cite{BT04,T17}.

\subsection{Linear evolutionary equation}\label{sec:expex}
Let us consider the following
three-dimensional 
Advection--Diffusion--Absorption evolutionary
equation, written in conservative form,
for a concentration $u(t,\bb x)$   (see~\cite{ZK99})

\begin{equation}\label{eq:ada}
  \left\{\begin{aligned}
  &\partial_t u(t,\bb x) +
  \sum_{\mu=1}^3\beta_\mu\partial_{x_\mu}( x_\mu u(t,\bb x))
  =\alpha\sum_{\mu=1}^3\beta_\mu^2\partial_{x_\mu}( x_\mu^2\partial_{x_\mu} u(t,\bb x))-\gamma
u(t,\bb x),\\
  &u(0,\bb x)=u_0(\bb x)=x_1(2-x_1)^2x_2(2-x_2)^2x_3(2-x_3)^2,
  \end{aligned}\right.
\end{equation}
where $\beta_\mu$, $\mu=1,2,3$, and $\alpha>0$ are advection and diffusion coefficients
and $\gamma\ge0$ is a coefficient governing the decay of $u(t,\bb x)$.
After a space discretization by second order centered finite differences
{\color{black}on a Cartesian grid},
we end up with a system of ODEs
\begin{equation}\label{eq:ode}
  \left\{
  \begin{aligned}
    \boldsymbol u'(t)&=(A_3\oplus A_2\oplus A_1)\boldsymbol u(t),\\
    \bu(0)&=\bu_0,
\end{aligned}\right.
  \end{equation}
where
$A_\mu\in\RR^{n_\mu\times n_\mu}$ is
the one-dimensional discretization of the operator
\begin{equation*}
  (2\alpha\beta_\mu^2x_\mu-\beta_\mu x_\mu)\partial_{x_\mu}+\alpha\beta_\mu^2 x_\mu^2\partial_{x_\mu^2}
  -\left(\beta_\mu+\frac{\gamma}{3}\right).
\end{equation*}
If we denote by $\boldsymbol U_0=\mathrm{vec}(\bu_0)$
and $\boldsymbol U(t)=\mathrm{vec}(\bu(t))$ the tensors
associated to the vectors $\boldsymbol u_0$ and $\boldsymbol u(t)$,
respectively, then we have
\begin{equation}\label{eq:exp3d}
  \boldsymbol U(t)=\boldsymbol U_0\times_1 \exp(t A_1)\times_2
\exp(t A_2) \times_3 \exp(t A_3).
\end{equation}
We consider equation~\eqref{eq:ada} for $\bx\in[0,2]^3$, coupled
with homogeneous Dirichlet--Neumann
conditions ($u(t,\bx)=0$ at {\color{black}$x_\mu=0$ and
  $\partial_{x_\mu}u(t,\bx)=0$ at $x_\mu=2$, $\mu=1,2,3$}).
The coefficients are fixed to
\begin{equation*}
  \beta_1=\beta_2=\beta_3=\frac{2}{3},\quad \alpha=\frac{1}{2},
  \quad \gamma=\frac{1}{100}.
\end{equation*}
Then, if we compute the needed matrix exponentials by the function
\verb+expm+ in \matlab{} and define
\begin{verbatim}

E{mu} = expm(tstar*A{mu});

\end{verbatim}
the solution $\bb U(t^*)$ at final time $t^*=0.5$ can be computed as
\begin{verbatim}

U = tucker(U0,E);

\end{verbatim}
since the matrix exponential is the exact solution and thus no substepping
strategy is needed.
The complete example is reported in the script \verb+example_exponential.m+.

In Table~\ref{tab:exp} we show the results with
a discretization in space of $\bb n=(50,55,60)$ 
grid points.
Since the problem is moderately stiff,
we consider for comparison the solution
of system~\eqref{eq:ode} by the \verb+ode23+ \matlab{} function
(which implements an explicit
adaptive Runge--Kutta method of order (2)3) and
by a standard implementation of the explicit Runge--Kutta method of order four
(RK4). For the Runge--Kutta methods, we consider both
the \emph{tensor} and the \emph{vector} implementations,
using the functions \verb+kronsumv+ and \verb+kronsum+, respectively
(see equivalence~\eqref{eq:actkron}).
The number of uniform time steps for RK4 has been chosen in order to obtain
a comparable error with respect to the result of the variable
time step solver \verb+ode23+.
As we can see,
the tensor formulation~\eqref{eq:exp3d} implemented using the
function \verb+tucker+ is much faster than any other considered approach.
Indeed, this is due to the fact that formula~\eqref{eq:exp3d}
requires a single time step and calls a level 3 BLAS routine
only three times. For other experiments involving the approximation
of the action of the matrix exponential
in tensor-structured problems,
we invite the reader to {\color{black}check~\cite{CCEOZ22}}.
\begin{table}[!ht]
  \centering
  \begin{tabular}{|c|c|c|c|c|}
    \hline
                 & Time steps & Elapsed time vector & Elapsed time tensor & Error\\
\hline
    \verb+tucker+ & 1 & -- & 0.03 & -- \\
    \verb+ode23+ & 1496 & 14.0 & 11.2 & 1.0e-4\\
    RK4          & 1351 & 9.14 & 6.33 & 3.7e-5\\
    \hline
  \end{tabular}
    \caption{Summary of the results for solving the ODEs system~\eqref{eq:ode}
    with the three described approaches.
    We report the number of time steps, the wall-clock times in seconds
    for both the tensor
    and the vector formulations (when feasible) and the relative error
    in infinity norm of the final solution with respect to the solution
    given by the \texttt{tucker} approach.}
  \label{tab:exp}
  \end{table}

\subsection{Semilinear evolutionary equation}\label{sec:precex}
We consider the following three-dimensional semilinear evolutionary equation
\begin{equation}\label{eq:semilinear}
  \left\{
  \begin{aligned}
    &\partial_t u(t,\bx) = \Delta u(t,\bx) + \frac{1}{1+u(t,\bx)^2} + \Phi(t,\bx),\\
    &u(0,\bx) = u_0(\bx)=x_1(1-x_1)x_2(1-x_2)x_3(1-x_3){\color{black},}
  \end{aligned}
  \right.
\end{equation}
for $\bx\in[0,1]^3$,
where the function $\Phi(t,\bx)$ is chosen so that the exact solution is
$u(t,\bx) = \rme^t u_0(\bx)$. We complete the equation with homogeneous Dirichlet
boundary conditions in all the directions. 
This is the three-dimensional generalization of the example
presented in \cite{HO05}.

We discretize the problem in space by means of second order centered finite 
differences {\color{black}on a Cartesian grid},
with $n_\mu$ grid points for the spatial variable $x_\mu$, 
$\mu=1,2,3$. Then, the application of the backward-forward Euler method 
leads to the following marching scheme
\begin{equation}\label{eq:semilin-disc}
  M\boldsymbol{u}_{k+1} = \boldsymbol{u}_k +
  \tau\boldsymbol{f}(t_k,\boldsymbol{u}_k){\color{black},}
\end{equation}
where $\boldsymbol{u}_{k}\approx u(t_k,\boldsymbol{x})$,
$\tau$ is the
time step size, $t_k$ is the current time and
\begin{equation*}
  \boldsymbol{f}(t_k,\boldsymbol{u}_k) =
  \frac{1}{1+\boldsymbol{u}_k^2} + \Phi(t_k,\boldsymbol{x}).
\end{equation*}
The matrix of the linear system~\eqref{eq:semilin-disc} is given by
\begin{equation*}
  M = M_3 \oplus M_2 \oplus M_1, \quad
  M_\mu = \left(\frac{1}{3}I_\mu - \tau A_\mu\right),
\end{equation*}
where $A_\mu$ is the discretization of the partial differential operator
$\partial_{x_\mu^2}$ and $I_\mu$ is the identity matrix of size $n_\mu$.
One could solve the linear system \eqref{eq:semilin-disc} using a direct
method, in particular by computing the Cholesky factors of the matrix $M$
once and for all (if the step size $\tau$ is constant). Another approach
would be to use the Conjugate Gradient (CG) method for
the single marching step~\eqref{eq:semilin-disc}. In  \matlab{}, the latter
can be performed as
\begin{verbatim}

pcg(M,uk+tau*f(tk,uk),tol,maxit,[],[],uk);

\end{verbatim}
or
\begin{verbatim}

pcg(Mfun,uk+tau*f(tk,uk),tol,maxit,[],[],uk);

\end{verbatim}
where \verb+M+ is the matrix assembled using \verb+kronsum+ (vector approach),
while \verb+Mfun+ is implemented by means of the function \verb+kronsumv+
(tensor approach).
As described in Section~\ref{sec:prec12d}, an effective preconditioner
for system~\eqref{eq:semilin-disc} is the one of ADI-type
\begin{equation*}
  P_3 \otimes P_2 \otimes P_1, \quad P_\mu = (I_\mu - \tau A_\mu).
\end{equation*}
The action of the inverse of this preconditioner on a vector $\bb v$ can be
easily performed in tensor formulation, see formula~\eqref{eq:preconnd},
and the resulting Preconditioned Conjugate Gradient (PCG) method
is
\begin{verbatim}

pcg(Mfun,uk+tau*f(tk,uk),tol,maxit,Pfun,[],uk);

\end{verbatim}
where \verb+Pfun+ is implemented through the \kP{} function \verb+itucker+.
The complete example is reported in the file \verb+example_imex.m+.

In Table~\ref{tab:semilin} we report the results obtained for a space  discretization
of $\bb n=(40,44,48)$ grid points. The time step size of the marching
scheme~\eqref{eq:semilin-disc} is $\tau=0.01$ and the final time of integration
is $t^*=1$. For all the methods, the final relative error with respect to the
exact solution measured in infinity norm is $9.7\cdot 10^{-3}$.
As it is clearly shown, the ADI-type preconditioner is really
effective in reducing the number of iterations of the CG method. Moreover,
the resulting method is the fastest among
all the considered approaches.

\begin{table}[ht!]
  \centering
  \begin{tabular}{|c|c|c|c|}
\hline
    &Avg.~iterations&  \multirow{2}{*}{\centering Elapsed time}\\
    & per time step & \\ 
     \hline
     Direct     & -- &  6.7 \\
     CG vector  & 30 &  3.3 \\
     CG tensor  & 30 &  2.2\\
     PCG tensor & 2  &  0.5\\
     \hline
  \end{tabular}
  \caption{Summary of the results for solving the semilinear
    equation~\eqref{eq:semilinear} by the method of lines
    and the backward--forward Euler method. The elapsed time is the
    wall-clock time measured in seconds.}
  \label{tab:semilin}
\end{table}
\section{Conclusions} \label{sec:conc}
In this work, we presented how it is possible to state
$d$-dimensional tensor-structured problems by means of composition of
one-dimensional rules, in such a way that the resulting $\mu$-mode
BLAS formulation 
can be efficiently implemented on modern computer hardware. The common thread 
consists in the suitable employment of tensor-product operations, with special 
emphasis on the Tucker operator and its variants.
{\color{black}
After validating our
package \kP{} against other commonly used tensor operation toolboxes,
  the effectiveness of the $\mu$-mode approach
compared to other well-established techniques
  is shown on several examples from different fields of 
numerical analysis. More in detail, we employed this approach} for 
a pseudospectral Hermite--Laguerre--Fourier trivariate function
decomposition, 
for the barycentric Lagrange interpolation of a {\color{black}five}-variate
function and for 
the numerical solution of three-dimensional stiff linear and semilinear 
evolutionary differential equations by means of exponential techniques and 
a (preconditioned) IMEX method, respectively.

\bmhead{Acknowledgments}
The authors acknowledge partial support from the Program Ricerca di Base
2019 of the University of Verona entitled ``Geometric Evolution of
Multi Agent Systems''.
Franco Zivcovich has received funding from the European Research Council (ERC) under the 
European Union’s Horizon 2020 research and innovation programme
(grant agreement No. 850941).
\section*{Declarations}
\bmhead{Data availability}
Data sharing not applicable to this
  article as no datasets were generated or analyzed during the current study.
  \bmhead{Conflict of interest}
  The authors declare that they have no conflict
  of interest.
  \begin{appendices}
\section{}\label{app:A}
Throughout the manuscript, the symbol $\otimes$ denotes the standard Kronecker
product of two matrices. In particular, given $A\in\mathbb{C}^{m\times n}$
and $B\in\mathbb{C}^{p\times q}$, we have
\begin{equation*}
A \otimes B = \left[
\begin{array}{ccc}
a_{11}B & \cdots & a_{1n}B \\
\vdots & \ddots & \vdots \\
a_{m1}B & \cdots & a_{mn}B
\end{array}\right] \in \CC^{mp\times nq}.
\end{equation*}

Moreover, we define the Kronecker sum of two matrices 
$A\in\mathbb{C}^{m\times m}$ and $B\in\mathbb{C}^{p\times p}$, denoted by the
symbol $\oplus$, as
\begin{equation*}
A \oplus B = A \otimes I_B + I_A \otimes B \in \CC^{mp \times mp} ,
\end{equation*}
where $I_A$ and $I_B$ are identity matrices of size $m$ and $p$, respectively.

We define also the vectorization operator,
denoted by $\mathrm{vec}$, which stacks
a given tensor $\bb T \in \CC^{m_1\times\cdots\times m_d}$ in a vector 
$\bb v \in \CC^{m_1\cdots m_d}$ in such a way that
\begin{equation*}
\mathrm{vec}(\bb T) = \bb v,\ \text{with}\ 
\bb v_{j} = \bb T(j_1,\ldots,j_d),
\quad j= j_1 + \sum_{\mu=2}^{d}(j_\mu-1)\prod_{k=1}^{\mu-1}m_k,
\end{equation*}
where $1 \leq j_\mu \leq m_\mu$ and $1 \leq \mu \leq d$.

The Kronecker product satisfies many properties, see \cite{L00} 
for a comprehensive review.
For convenience of the reader, we list here the relevant ones in our context 

\begin{enumerate}
    \item $A \otimes (B_1+B_2) = A \otimes B_1 + A \otimes B_2$
    for every $A\in\mathbb{C}^{m\times n}$ and $B_1,B_2 \in\mathbb{C}^{p\times q}$;
    \item $(B_1+B_2) \otimes A = B_1 \otimes A + B_2 \otimes A$
    for every $B_1,B_2 \in\mathbb{C}^{p\times q}$ and $A\in\mathbb{C}^{m\times n}$;
    \item $(\lambda A) \otimes B = A \otimes (\lambda B) = \lambda (A \otimes B)$
    for every $\lambda\in\CC$, $A\in\mathbb{C}^{m\times n}$ and
    $B \in\mathbb{C}^{p\times q}$;
    \item  $(A\otimes B) \otimes C = A \otimes (B \otimes C)$ for every
    $A\in\mathbb{C}^{m\times n}$, $B \in\mathbb{C}^{p\times q}$ and
    $C \in\mathbb{C}^{r\times s}$;
    \item $(A\otimes B)^{\sf T} = A^{\sf T}\otimes B^{\sf T}$ for every 
          $A\in\mathbb{C}^{m\times n}$ and $B \in\mathbb{C}^{p\times q}$;
    \item $(A\otimes B)^{-1} = A^{-1}\otimes B^{-1}$ for every invertible matrix
          $A\in\mathbb{C}^{m\times m}$ and $B \in\mathbb{C}^{p\times p}$;
    \item $(A \otimes B)(D \otimes E) = (AD)\otimes(BE)$ for every 
    $A\in\mathbb{C}^{m\times n}$, $B \in\mathbb{C}^{p\times q}$,
    $D \in\mathbb{C}^{n\times r}$ and $E \in\mathbb{C}^{q\times s}$;
    \item $\mathrm{vec}(ADC) = (C^{\sf T}\otimes A)\mathrm{vec}(D)$ for every
    $A\in\mathbb{C}^{m\times n}$, $D \in\mathbb{C}^{n\times r}$ and
    $C \in\mathbb{C}^{r\times s}$.
\end{enumerate}

\end{appendices}

  \bibliography{biblio_mumode}

\end{document}